\documentclass[12pt]{article}
\usepackage{amssymb}
\usepackage{graphics}

\newcommand{\qed}{$\;\;\;\Box$}
\newenvironment{proof}{\par\smallbreak{\sl\bf Proof.~}}
{\unskip\nobreak\hfill \qed \par\medbreak}

\newcounter{claim}
\renewcommand{\theclaim}{\arabic{claim}}
{\par\medskip\par}

{\qed\par\smallbreak}
\newcommand{\N}{{\mathbb N}}
\newcommand{\R}{{\mathbb R}}

  \newcommand{\Z}{{\mathbb Z}}

\newcommand{\LL}{{\cal L}}

\newcommand{\beq}{\begin{equation}}
\newcommand{\ee}{\end{equation}}

\renewcommand{\d}{\partial}

\newtheorem{thm}{Theorem}[section]
\newtheorem{lem}[thm]{Lemma}

\newtheorem{cor}[thm]{Corollary}
\newtheorem{rem}[thm]{Remark}

\newcommand{\al}{\alpha}
\newcommand{\be}{\beta}
\newcommand{\ga}{\gamma}

\newcommand{\vphi}{\varphi}

\newcommand{\om}{\omega}

\title{Fredholm Solvability of Periodic Neumann Problem for a Linear Telegraph Equation
} 

\newcounter{thesame}
\author{
I.~Kmit\thanks{Supported by a Humboldt Research Fellowship.} \\
{\small
Institute of Mathematics, Humboldt University of Berlin,}
\\
{\small Rudower Chaussee 25, D-12489 Berlin, Germany }
\\
{\small
and Institute for Applied Problems of Mechanics and Mathematics, }
\\
{\small
Ukrainian Academy of Sciences,  Naukova St.\ 3b, 79060 Lviv,
Ukraine 
}
\\
{\small   E-mail:
{\tt kmit@informatik.hu-berlin.de}}
}
 
\begin{document}

\maketitle

\begin{abstract}
\noindent
We investigate  the linear telegraph equation 
$$
u_{tt}-u_{xx}+2\mu u_t=f(x,t)
$$
with periodic Neumann boundary conditions. We prove that the operator of the problem is
modeled as a Fredholm operator of index zero in the scale of Sobolev-type spaces of periodic functions.
This result extends to small perturbations of the equation where $\mu$ becomes variable and discontinuous or an
additional zero-order term appears. We also show that the solutions to the problem are smoothing.
\end{abstract}

\emph{Key words:} telegraph equation, correlated random walk system, periodic problems, Neumann boundary conditions, 
Fredholm solvability, smoothing solutions

\emph{Mathematics Subject Classification: 35L20, 35B10}

\section{Introduction}\label{sec:intr} 

The {\it telegraph equation} 
\beq\label{eq:nonlinear}
u_{tt}-u_{xx}+2\mu u_t+F(x,t,u)=0,\quad (x,t)\in(0,1)\times\R
\ee
combines features of  diffusion and wave equations. It describes dissipative wave processes, e.g., 
in  transmission and propagation of  electrical signals, dynamical processes in biological populations, 
economical and ecological systems etc. (see~\cite{Horst,Mark} and references therein). Hillen~\cite{Hil_Hopf}
discusses the appearance of Hopf bifurcations for (\ref{eq:nonlinear}) with Neumann boundary conditions in the case when $\mu$
is a negative constant. An important step towards a rigorous bifurcation analysis 
(via the Implicit Function Theorem and the Lyapunov-Schmidt reduction~\cite{ChowHale,Ki}) is to 
establish the Fredholm solvability of a linearized problem.
Note in this respect that the Fredholm property for hyperbolic PDEs is much less studied than for ODEs and parabolic PDEs.

We will investigate a linearization of (\ref{eq:nonlinear}) known as the {\it damped wave equation}:
\beq\label{eq:tel}
u_{tt}-u_{xx}+2\mu u_t=f(x,t),\quad (x,t)\in(0,1)\times\R,
\ee
where $\mu$ is a constant. This equation describes a correlated random walk  under the assumption that
particles with density $u$ have a constant speed and a constant turning rate $\mu$. Furthermore, we impose 
 time-periodicity conditions
\beq\label{eq:per}
\begin{array}{l}
\displaystyle
u\left(x,t+\frac{2\pi}{\om}\right) = u(x,t), \quad (x,t)\in[0,1]\times\R\\
\displaystyle
u_t\left(x,t+\frac{2\pi}{\om}\right) = u_t(x,t), \quad (x,t)\in[0,1]\times\R,
\displaystyle
\end{array}
\ee
where $\om>0$ is a particle speed, and Neumann boundary conditions
\beq\label{eq:Neu}
u_x(0,t) = u_x(1,t)=0, \quad t\in\R.
\ee
Solvability of periodic problems for the telegraph equation  (\ref{eq:tel}) is investigated, in particular, 
in~\cite{Mawhin_per-Neu,Maw_per2,Hil_Hopf,Kim,Kim1,Maw_per1}. Our main result (Theorem~\ref{thm:fredh_tel})
is the Fredholm Alternative  for the problem (\ref{eq:tel})--(\ref{eq:Neu}). Previously, the Fredholm zero index property 
was known only in the double-periodic case~\cite{Kim,Kim1}.
We prove our result by reduction to a related periodic Neumann problem for a hyperbolic system. 
When splitting $u=v+w$ into the density $v$ of particles moving  right and the density
$w$ of particles moving left,  (\ref{eq:tel})--(\ref{eq:Neu}) is brought into the form:
\beq\label{eq:walk1}
\begin{array}{l}
v_t+v_x=g(x,t)+\mu(w-v),\quad (x,t)\in(0,1)\times\R\\
\displaystyle
w_t-w_x=g(x,t)+\mu(v-w),\quad (x,t)\in(0,1)\times\R,
\end{array}
\ee
\beq\label{eq:per1}
\begin{array}{l}
\displaystyle
v\left(x,t+\frac{2\pi}{\om}\right) = v(x,t), \quad (x,t)\in[0,1]\times\R\\
\displaystyle
w\left(x,t+\frac{2\pi}{\om}\right) = w(x,t), \quad (x,t)\in[0,1]\times\R,
\end{array}
\ee
\beq\label{eq:Neu1}
\begin{array}{l}
v(0,t) = w(0,t), \quad t\in\R
\\
\displaystyle
v(1,t) = w(1,t), \quad t\in\R.
\end{array}
\ee
System (\ref{eq:walk1}) describes a random walk process introduced by Taylor \cite{Taylor}, where
a particle moving right dies with rate $\mu$ and is reborn as a particle moving left
with the same rate. Homogeneous Neumann boundary  conditions (\ref{eq:Neu1}) describe reflection of particles 
from the boundary.

For (\ref{eq:walk1})--(\ref{eq:Neu1}) we proved a Fredholm Alternative in~\cite{KmRe1}. Now, 
we show the equivalence between the models  (\ref{eq:tel})--(\ref{eq:Neu}) and (\ref{eq:walk1})--(\ref{eq:Neu1})
in a certain functional analytic sense, which allows us to derive the Fredholm property for 
 (\ref{eq:tel})--(\ref{eq:Neu}) from the result about (\ref{eq:walk1})--(\ref{eq:Neu1}) obtained in~\cite{KmRe1}.

By a perturbation argument, our results extend to the equation
\beq\label{eq:tel_perturb}
u_{tt}-u_{xx}+\nu(x,t)u_t+\al(x,t)u=f(x,t),\quad (x,t)\in(0,1)\times\R,
\ee
where $\nu(x,t)-2\mu$ and $\al(x,t)$ are sufficiently small in appropriate function spaces (see Remark~\ref{rem:perturb}).

In Section~\ref{sec:spaces} we introduce function spaces of solutions and give account of their useful properties.
Equivalence of the models (\ref{eq:tel})--(\ref{eq:Neu}) and (\ref{eq:walk1})--(\ref{eq:Neu1}) is proved 
in~Section~\ref{sec:equiv}. Finally, in Section~\ref{sec:fredh} we prove our main result (Theorem~\ref{thm:fredh_tel}).

\section{Function spaces and operators}~\label{sec:spaces}

We will use yet  another representation of the problem
(\ref{eq:walk1})--(\ref{eq:Neu1}), namely,
\beq\label{eq:walk2}
\begin{array}{l}
u_t+z_x=g(x,t),\quad (x,t)\in(0,1)\times\R\\
\displaystyle
z_t+u_x=-2\mu z,\quad (x,t)\in(0,1)\times\R,
\end{array}
\ee
\beq\label{eq:per2}
\begin{array}{l}
\displaystyle
u\left(x,t+\frac{2\pi}{\om}\right) = u(x,t), \quad (x,t)\in[0,1]\times\R\\
\displaystyle
z\left(x,t+\frac{2\pi}{\om}\right) = z(x,t), \quad (x,t)\in[0,1]\times\R,
\end{array}
\ee
\beq\label{eq:Neu2}
z(0,t) = z(1,t)=0, \quad t\in\R,
\ee
where
\beq\label{eq:uv}
u=\frac{v+w}{2},\qquad z=\frac{v-w}{2}.
\ee
Here $2u$ stands for the total particle density, while $2z$  for particle flux.
The relationship between the models (\ref{eq:tel})--(\ref{eq:Neu}), (\ref{eq:walk1})--(\ref{eq:Neu1}), and (\ref{eq:walk2})--(\ref{eq:Neu2})
is discussed in Section~\ref{sec:equiv}. We will show that they are equivalent in a certain sense.

For solutions and right hand sides of the problems (\ref{eq:tel})--(\ref{eq:Neu}), (\ref{eq:walk1})--(\ref{eq:Neu1}), and (\ref{eq:walk2})--(\ref{eq:Neu2})
we now introduce   pairs of spaces $(U_b^\ga,H^{0,\ga-1})$,
$(V_b^\ga,W_d^\ga)$, and $(Z_b^\ga,W_0^\ga)$, respectively. Here $\ga\ge 1$ denotes a real scaling parameter.
The subscript $b$ indicates that we construct spaces constrained by boundary conditions, $d$
stands for the diagonal subspace of pairs $(u,u)$, and $0$ for the subspace of pairs $(u,0)$.

Given $l\in\N_0$ and $\ga\ge 0$, we first introduce the space $H^{l,\ga}$   of all
measurable functions $u: (0,1)\times\R\to\R$ such that
$$
u(x,t)=u\left(x,t+\frac{2\pi}{\om}\right) \,\,\mbox{for\,\,a.a. }\,
\,(x,t)\in[0,1]\times\R
$$
and 
\begin{equation}\label{eq:1.3}
\|u\|_{H^{l,\ga}}^2:=\sum\limits_{k\in\Z}(1+k^2)^{\gamma}
\sum\limits_{m=0}^l\int\limits_0^1\left|\int\limits_0^{\frac{2\pi}{\om}}
\partial_x^mu(x,t)e^{-ik\om t}\,dt\right|^2\,dx<\infty.\nonumber
\end{equation}
It is well-known (see, e.g., \cite[Chapter 2.4]{vejvoda})
that $H^{l,\ga}$ is a Banach space.
In fact, this is the space of all $\frac{2 \pi}{\om}$-periodic maps $u: \R \to H^l(0,1)$ that are 
locally $L^2$-Bochner
integrable together with their generalized derivatives up to the (possibly noninteger) order $\ga$.
Furthermore, we define
\begin{eqnarray*}
W^{\gamma}&=&H^{0,\ga}\times H^{0,\ga},\\
W_d^{\gamma}&=&\left\{(g,f)\in W^{\gamma}:\,g=f\right\},\\
W_0^{\gamma}&=&\left\{(g,f)\in W^{\gamma}:\,f=0\right\},\\
V^{\gamma}&=& \left\{(v,w)\in W^{\gamma}:\,
(v_t+v_x,w_t-w_x)\in W^{\gamma}\right\},\\
Z^{\gamma}&=& \left\{(u,z)\in W^{\gamma}:\,
(u_t+z_x,z_t+u_x)\in W^{\gamma}\right\},\\
U^{\gamma}&=& \left\{u\in H^{0,\gamma}:\,
u_x\in H^{0,\gamma-1},u_{tt}-u_{xx}\in H^{0,\gamma-1}\right\},
\end{eqnarray*}
where $u_t$, $u_x$, $u_{tt}$,  $u_{xx}$, $v_t$,  $v_x$, $z_t$, and  $z_x$ are understood in the sense
of generalized derivatives.
The function spaces $W^{\ga}$,  $V^{\gamma}$, $Z^{\gamma}$, and $U^{\gamma}$ will be endowed with the norms
\begin{eqnarray*}
\|(g,f)\|_{W^{\gamma}}^2&=&\|g\|_{H^{0,\ga}}^2+\|f\|_{H^{0,\ga}}^2,
\\
\|(v,w)\|_{V^{\gamma}}^2&=&\|(v,w)\|_{W^{\gamma}}^2
+\|(v_t+v_x,w_t-w_x)\|_{W^{\gamma}}^2,
\\
\|(u,z)\|_{Z^{\gamma}}^2&=&\|(u,z)\|_{W^{\gamma}}^2
+\|(u_t+z_x,z_t+u_x)\|_{W^{\gamma}}^2,\\
\|u\|_{U^{\gamma}}^2&=&\|u\|_{H^{0,\gamma}}^2+\|u_x\|_{H^{0,\gamma-1}}^2
+\|u_{tt}-u_{xx}\|_{H^{0,\gamma-1}}^2.
\end{eqnarray*}

In the following two lemmas we collect some useful properties of the function spaces $V^\ga$ and $U^\ga$, 
respectively.

\begin{lem} \cite[Section~2]{KmRe1}\label{lem:V}

(i) The space $V^{\gamma}$ is complete.

(ii) If $\gamma\ge 1$, then  $V^{\gamma}$ is continuously embedded into $\left(H^{1,\ga-1}\right)^2$.

(iii)  For any $x \in [0,1]$ there exists a continuous trace map
$$
(v,w) \in V^\ga \mapsto \Bigl(v(x,\cdot),w(x,\cdot)\Bigr)
\in \left(L^2\left(0,\frac{2\pi}{\om}\right)\right)^2.
$$
\end{lem}

Similar properties are encountered in the function spaces $U^\gamma$.

\begin{lem} \label{lem:U}

(i) The space $U^{\gamma}$ is complete.

(ii) If $\gamma\ge 2$, then  $U^{\gamma}$ is continuously embedded into $H^{2,\ga-2}$.

(iii)  If $\ga\ge 1$, then for any $x \in [0,1]$ there exists a continuous trace map
$$
u \in U^\ga \mapsto u_x(x,\cdot)
\in L^2\left(0,\frac{2\pi}{\om}\right).
$$
\end{lem}

\begin{proof}
{\bf (i)}
Let $(u_{j})$ be a fundamental sequence
in $U^{\gamma}$.
Then $(u_j)$ is fundamental in $H^{0,\ga}$ and $(\d_xu_j)$ and $(\d_t^2u_j-\d_x^2u_j)$
are fundamental in $H^{0,\gamma-1}$. Since $H^{0,\gamma}$ is complete,
there exist
$u\in H^{0,\gamma}$ and $v,w\in H^{0,\gamma-1}$ such that
$$
u_{j}\to u \mbox{ in } H^{0,\ga},\quad
\d_xu_j\to  v  \mbox{ in } H^{0,\ga-1},
 \mbox{ and } \quad
\d_t^2u_j-\d_x^2u_j\to w  \mbox{ in } H^{0,\ga-1}
$$
 as $j\to\infty$.
It remains to show that
$
\d_xu=v
$
and
$
\d_t^2u-\d_x^2u=w
$
in the sense of generalized derivatives. For this purpose take a smooth
function
$\vphi: (0,1)\times\left(0,\frac{2\pi}{\om}\right)\to\R$ with compact
support and note that
\begin{eqnarray*}
&\displaystyle
-\int\limits_{0}^{\frac{2\pi}{\om}}\int\limits_{0}^1u\d_x\vphi\,
dx\,dt=
-\lim\limits_{j\to\infty}\int\limits_{0}^{\frac{2\pi}{\om}}\int\limits_{
0}^1
u_j\d_x\vphi\,dx\,dt&\\
&\displaystyle
=\lim\limits_{j\to\infty}\int\limits_{0}^{\frac{2\pi}{\om}}\int\limits_{
0}^1
\d_xu_j\vphi\,dx\,dt=
\int\limits_{0}^{\frac{2\pi}{\om}}\int\limits_{0}^1
v\vphi\,dx\,dt.&
\end{eqnarray*}
Similarly, 
\begin{eqnarray*}
&\displaystyle
\int\limits_{0}^{\frac{2\pi}{\om}}\int\limits_{0}^1u(\d_t^2-\d_x^2)\vphi\,
dx\,dt=
\lim\limits_{j\to\infty}\int\limits_{0}^{\frac{2\pi}{\om}}\int\limits_{
0}^1
u_j(\d_t^2-\d_x^2)\vphi\,dx\,dt&\\
&\displaystyle
=\lim\limits_{j\to\infty}\int\limits_{0}^{\frac{2\pi}{\om}}\int\limits_{
0}^1
(\d_t^2-\d_x^2)u_j\vphi\,dx\,dt=
\int\limits_{0}^{\frac{2\pi}{\om}}\int\limits_{0}^1
w\vphi\,dx\,dt.&
\end{eqnarray*}

{\bf (ii)} 
Take $u\in U^{\ga}$. Then $u\in H^{0,\ga}$ and,
hence, $\d_t^2u\in H^{0,\ga-2}$.
By the definition of the space $U^{\ga}$, we have $\d_x^2u\in H^{0,\ga-2}$ as well.
Therefore, $u\in H^{2,\ga-2}$. Moreover, we have
\begin{eqnarray*}
\lefteqn{
\|u\|_{H^{2,\ga-2}}^2=\|u\|_{H^{0,\ga-2}}^2+\|\d_xu\|_{H^{0,\ga-2}}^2+
\|\d_x^2u\|_{H^{0,\ga-2}}^2}\\
&&\le
\|u\|_{H^{0,\ga-2}}^2+c\|\d_xu\|_{H^{0,\ga-1}}^2+
c\|\d_x^2u\|_{H^{0,\ga-2}}^2 \\
&&\le
\|u\|_{H^{0,\ga-2}}^2+c\|\d_xu\|_{H^{0,\ga-1}}^2+
c\|\d_x^2u-\d_t^2u\|_{H^{0,\ga-2}}^2
+c\|\d_t^2u\|_{H^{0,\ga-2}}^2\le C\|u\|^2_{U^{\gamma}},
\end{eqnarray*}
where the constants $c$ and $C$ do not depend on $u$.

Claim {\bf (iii)} follows from the definition of $U^\ga$.
\end{proof}

\begin{rem}\label{rem:trace}
By  (\ref{eq:uv}) and Lemma~\ref{lem:V} (iii), for any $x \in [0,1]$ there exists a continuous trace map
$$
(u,z) \in Z^\ga \mapsto \Bigl(u(x,\cdot),z(x,\cdot)\Bigr)
\in \left(L^2\left(0,\frac{2\pi}{\om}\right)\right)^2.
$$
\end{rem}

Lemmas \ref{lem:V} (iii) and  \ref{lem:U} (iii) and Remark~\ref{rem:trace} motivate
consideration of the following closed subspaces in $V^\ga$, $Z^\ga$, and $U^\ga$:
\begin{eqnarray*}
V_b^\ga&=&\{(v,w) \in V^\ga: \; (\ref{eq:Neu1}) \mbox{ is fulfilled for a.a. } t \in \R\},\\
Z_b^\ga&=&\{(u,z) \in Z^\ga: \; (\ref{eq:Neu2}) \mbox{ is fulfilled for a.a. } t \in \R\},\\
U_b^\ga&=&\{u \in U^\ga: \; (\ref{eq:Neu}) \mbox{ is fulfilled for a.a. } t \in \R\}.
\end{eqnarray*}
Finally, we introduce linear operators
$L_{WS},\tilde L_{WS}\in\LL(V_b^{\gamma};W_d^{\gamma})$ by
$$
L_{WS}
\left[
\begin{array}{c}
v\\w
\end{array}
\right]
=
\left[
\begin{array}{c}
v_t+v_x-\mu(w-v)\\w_t-w_x-\mu(v-w)
\end{array}
\right], 
$$
$$
\tilde{L}_{WS}
\left[
\begin{array}{c}
v\\w
\end{array}
\right]
=
\left[
\begin{array}{c}
-v_t-v_x-\mu(w-v)\\
-w_t+w_x-\mu(v-w)
\end{array}
\right],
$$
a linear operator $L_{WS}^\prime\in\LL(Z_b^{\gamma};W_0^{\gamma})$ by
$$
L_{WS}^\prime
\left[
\begin{array}{c}
u\\z
\end{array}
\right]
=
\left[
\begin{array}{c}
u_t+z_x\\z_t+u_x+2\mu z
\end{array}
\right], 
$$
and  linear operators $L_{TE},\tilde L_{TE}\in\LL(U_b^{\gamma};H^{0,\gamma-1})$ by
$$
L_{TE}(u)
=
u_{tt}-u_{xx}+2\mu u_t, 
$$
$$
\tilde L_{TE}(u)
=
u_{tt}-u_{xx}-2\mu u_t.
$$

\section{Equivalence of the models}\label{sec:equiv}

We first describe the reduction of the problem (\ref{eq:walk1})--(\ref{eq:Neu1}) to
(\ref{eq:tel})--(\ref{eq:Neu}) that was suggested in \cite{Kac}, see also \cite{Hillen}. 
Recall that the simple change of variables (\ref{eq:uv})
 transforms the system (\ref{eq:walk1})--(\ref{eq:Neu1}) to the form (\ref{eq:walk2})--(\ref{eq:Neu2}).
Now, assuming two-times differentiability of $z$ and $w$, we differentiate the first equation in
(\ref{eq:walk2}) with respect to $t$ and the second equation with respect to $x$. After a simple calculation 
we come to the problem (\ref{eq:tel})--(\ref{eq:Neu}) with $f=g_t+2\mu g$.

Formally, we will show that the two problems (\ref{eq:tel})--(\ref{eq:Neu}) and (\ref{eq:walk1})--(\ref{eq:Neu1}) 
are equivalent in the following sense: there exist isomorphisms $\al_1: V_b^\ga\to Z_b^\ga$,
$\al_2: Z_b^\ga\to U_b^\ga$ and $\be_1: W_d^\ga\to W_0^\ga$, $\be_2: W_0^\ga\to H^{0,\ga-1}$ between 
the respective linear spaces such that the diagram
\beq\label{eq:diag}
\begin{array}{ccc} V_b^\ga & \stackrel{L_{WS}}{\longrightarrow} & W_d^\ga \\
\al_1\Big\downarrow       &                      & \Big\downarrow\be_1\\
Z_b^\ga & \stackrel{L_{WS}^\prime}{\longrightarrow} & W_0^\ga \\
\al_2\Big\downarrow       &                      & \Big\downarrow\be_2\\
U_b^\ga    & \stackrel{L_{TE}}{\longrightarrow} & H^{0,\ga-1}  
     \end{array} 
\ee
is commutative, that is,
\beq\label{eq:comm}
\be_1\circ L_{WS}=L_{WS}^\prime\circ\al_1,\quad \be_2\circ L_{WS}^\prime=L_{TE}\circ\al_2.
\ee

Specifically, we define $\al_1, \be_1, \al_2$, and $\be_2$ by 
\beq\label{eq:ab1}
\al_1(v,w)=\left(\frac{v+w}{2},\frac{v-w}{2}\right),\quad \be_1(g,g)=(g,0),
\ee
\beq\label{eq:ab2}
\al_2(u,z)=u,\quad  \be_2(g,0)=g_t+2\mu g.
\ee
Furthermore, let
\beq\label{eq:ab}
\al=\al_2\circ\al_1,\quad \be=\be_2\circ\be_1.
\ee
The following result shows the commutativity of the upper part of the diagram (\ref{eq:diag})
or, the same, the equivalence of the problems (\ref{eq:walk1})--(\ref{eq:Neu1}) and (\ref{eq:walk2})--(\ref{eq:Neu2}). 
Its proof is straightforward.
\begin{lem}\label{lem:eq1}
 Suppose $\ga\ge 2$. Then

(i) The maps $\al_1: V_b^\ga\to Z_b^\ga$ and $\be_1: W_d^\ga\to W_0^\ga$ defined by
 (\ref{eq:ab1}) are isomorphisms.

(ii) $\be_1\circ L_{WS}=L_{WS}^\prime\circ\al_1$.
\end{lem}

To prove the commutativity of the lower part of the diagram (\ref{eq:diag}), i.e.,
the equivalence of the problems (\ref{eq:tel})--(\ref{eq:Neu}) and
(\ref{eq:walk2})--(\ref{eq:Neu2}), we will need the following simple lemma.

\begin{lem}\label{lem:fg}
If $\mu\ne 0$, then the map $f\in H^{0,\ga-1}\mapsto g\in H^{0,\ga}$ where
\begin{equation}\label{eq:fg}
\begin{array}{cc}
g_t+2\mu g=f\\
\displaystyle
g\left(x,t+\frac{2\pi}{\om}\right) = g(x,t)
\end{array}
\end{equation}
is bijective.
\end{lem}

\begin{proof}
The problem (\ref{eq:fg}) has a unique solution given by the formula
\begin{equation}\label{eq:g}
g(x,t) = \frac{e^{-2\mu \left(t+\frac{2\pi}{\om}\right)}}{1-e^{-2\mu\frac{2\pi}{\om}}}\int\limits_0^{\frac{2\pi}{\om}}e^{2\mu\tau}
f(x,\tau)\,d\tau-\int\limits_0^te^{2\mu\left(\tau-t\right)}
f(x,\tau)\,d\tau,
\end{equation}
which gives us the lemma.
\end{proof}

\begin{cor}\label{cor:uv}
If $\mu\ne 0$, then the map $u_x\in H^{0,\ga-1}\mapsto z\in H^{0,\ga}$ where
\begin{equation}\label{eq:u_xv}
\begin{array}{cc}
z_t+2\mu z=u_x\\\displaystyle
z\left(x,t+\frac{2\pi}{\om}\right) = z(x,t)\nonumber
\end{array}
\end{equation}
is bijective. Furthermore,
\begin{equation}\label{eq:v}
z(x,t) = \frac{e^{-2\mu \left(t+\frac{2\pi}{\om}\right)}}{1-e^{-2\mu\frac{2\pi}{\om}}}\int\limits_0^{\frac{2\pi}{\om}}e^{2\mu\tau}
u_x(x,\tau)\,d\tau-\int\limits_0^te^{2\mu\left(\tau-t\right)}
u_x(x,\tau)\,d\tau
\end{equation}
\end{cor}

Now we state the desired equivalence result.

\begin{lem}\label{lem:eq2} Suppose $\ga\ge 2$. Then

(i) The maps $\al_2: Z_b^\ga\to U_b^\ga$ and $\be_2: W_0^\ga\to H^{0,\ga-1}$ defined by
 (\ref{eq:ab2}) are isomorphisms.

(ii) $\be_2\circ L_{WS}^\prime=L_{TE}\circ\al_2$.
\end{lem}

\begin{proof}
{\bf(i)} We first show that $\al_2$ maps $Z_b^\ga$ to $U_b^\ga$ and $\be_2$ maps $W_0^\ga$
to $H^{0,\ga-1}$. The latter is obvious. To show the former, 
suppose that $(u,z)\in Z_b^{\gamma}$ solves (\ref{eq:walk2})--(\ref{eq:Neu2}). 
By the definition of $Z^\ga$, $z_t\in H^{0,\ga-1}$ and $z_t+u_x\in H^{0,\ga-1}$, hence
$u_x\in H^{0,\ga-1}$.
Since $z\in H^{1,\ga-1}$, we have $z_t\in H^{1,\ga-2}$. 
Now, as $z_t+u_x\in H^{0,\ga}$, we have  $u_x\in H^{1,\ga-2}$ and, hence, $u_{xx}\in H^{0,\ga-2}$.
Therefore $z_{tx}+u_{xx}\in  H^{0,\ga-2}$ and $u_{tt}+z_{xt}\in  H^{0,\ga-2}$.
Consequently, $u_{tt}-u_{xx}\in  H^{0,\ga-2}$ as well. This implies that $u\in U^\ga$.
To check the boundary conditions  (\ref{eq:per}) and (\ref{eq:Neu}), 
we take into account part (iii) of Lemma~\ref{lem:U} about the traces of $u$.
Now, conditions (\ref{eq:Neu}) follow from (\ref{eq:Neu2}) and the second equation of (\ref{eq:walk2}).
Conditions (\ref{eq:per}) are a straightforward consequence of (\ref{eq:per2}).

Further set 
$$
\al_2^{-1}(u)=(u,z) \mbox{ where } z \mbox{ is given by } (\ref{eq:v})
$$ 
and 
$$
\be_2^{-1}(f)=(g,0) \mbox{ where } g \mbox{ is given by } (\ref{eq:g}).
$$ 
We are done if we show that
\beq\label{eq:maps}
\al_2^{-1} \mbox{ maps } U_b^\ga \mbox{ into } Z_b^\ga  \mbox{ and }
\be_2^{-1} \mbox{ maps } H^{0,\ga-1} \mbox{ into } W_0^\ga
\ee
and that
\beq\label{eq:comp}
\be_2^{-1}\circ\be_2=I_{W_0^\ga},\quad \be_2\circ\be_2^{-1}=I_{H^{0,\ga-1}},\quad 
\al_2\circ\al_2^{-1}=I_{U^\ga_b},\quad \al_2^{-1}\circ\al_2=I_{Z^\ga_b}.
\ee

We start with proving (\ref{eq:maps}).
Since the right-hand side of the representation (\ref{eq:v}) belongs to $H^{0,\ga}$, we have $\d_tz\in H^{0,\ga-1}$. Differentiating 
(\ref{eq:v})  with respect to $t$, we easily arrive at the second equality of (\ref{eq:walk2}). To meet the first equality, we start
from the weak formulation of (\ref{eq:tel})--(\ref{eq:Neu}):
By Lemma~\ref{lem:fg}, any function $f\in H^{0,\ga-1}$ admits a unique representation $f=g_t+2\mu g$ where 
$g\in H^{0,\ga}$. On the account of this fact, for any $\vphi\in C^{1}\left([0,1]\times[0,\frac{2\pi}{\om}]\right)$
with $\vphi\left(x,t+\frac{2\pi}{\om}\right)=\vphi(x,t)$ we have
\begin{eqnarray*}
\lefteqn{
0=
\int\limits_{0}^{\frac{2\pi}{\om}}\int\limits_{0}^1\left[-u_{tt}\vphi-u_x\vphi_x-2\mu u_t\vphi+(g_t+2\mu g)\vphi\right]\,dxdt
}\\&&
=\int\limits_{0}^{\frac{2\pi}{\om}}\int\limits_{0}^1\left[u_{t}\vphi_t- g\vphi_t+[z_t+2\mu z]\vphi_x-2\mu u_t\vphi+2\mu g\vphi\right]\,dxdt
\\&&
=\int\limits_{0}^{\frac{2\pi}{\om}}\int\limits_{0}^1\left[u_{t}\vphi_t-g\vphi_t+z_x\vphi_t-2\mu z_x\vphi-2\mu u_t\vphi+2\mu g\vphi\right]\,dxdt
\\&&
=\int\limits_{0}^{\frac{2\pi}{\om}}\int\limits_{0}^1\left[(u_{t}+z_x-g)\vphi_t-2\mu( z_x+u_t- g)\vphi\right]\,dxdt
\\&&
=\int\limits_{0}^{\frac{2\pi}{\om}}\int\limits_{0}^1\left[(u_{t}+z_x-g)(\vphi_t-2\mu\vphi)\right]\,dxdt.
\end{eqnarray*}
Taking a constant $\vphi$, we get
\beq\label{eq:psi}
\int\limits_{0}^{\frac{2\pi}{\om}}\int\limits_{0}^1(u_{t}+z_x-g)\,dxdt=0.
\ee
It follows that  
$$
u_{t}+z_x-g=0\quad \mbox{ a.e.  on}\quad (0,1)\times\left(0,\frac{2\pi}{\om}\right).
$$
The first equality of  (\ref{eq:walk2}) is  therefore fulfilled. Furthermore, the  system (\ref{eq:walk2}) implies that $(u,z)\in Z^\ga$.
The boundary conditions (\ref{eq:per2}) and (\ref{eq:Neu2}) follow from 
 (\ref{eq:per}), (\ref{eq:Neu}),  and (\ref{eq:walk2}). 

To finish this part of the proof, it remains to note that the first two equalities in (\ref{eq:comp}) 
follow by Lemma~\ref{lem:fg}, the third equality  is straightforward, and the 
last one  follows from (\ref{eq:v}) and the second equality in 
(\ref{eq:walk2}).

{\bf (ii)} Differentiating now the first equality of (\ref{eq:walk2}) with respect to $t$ and
the second  with respect to $x$, subtracting the resulting equations and then
substituting $z_x$ from the first equation of (\ref{eq:walk2}), we come to 
(\ref{eq:tel}) with $f=g_t+2\mu g$. 
Hence $\left(L_{TE}\circ\al_2\right)(u,z)=g_t+2\mu g$.
Moreover, by definitions of $L_{WS}^\prime$ and $\be_2$, we have  $L_{WS}^\prime(u,z)=(g,0)$
and $\be_2(g,0)=g_t+2\mu g$. The desired assertion follows.

The proof is complete.
\end{proof}

We are prepared to formulate the main result of this section about 
the equivalence of the random walk problem (\ref{eq:walk1})--(\ref{eq:Neu1}) and the telegraph problem (\ref{eq:tel})--(\ref{eq:Neu}).

\begin{thm}\label{thm:eq} Suppose $\ga\ge 2$. Then

(i)  The maps $\al: V_b^\ga\to U_b^\ga$ and $\be: W_d^\ga\to H^{0,\ga-1}$ defined by
 (\ref{eq:ab}) are isomorphisms.

(ii) $\be\circ L_{WS}=L_{TE}\circ\al$.
\end{thm}
The theorem follows directly from Lemmas~\ref{lem:eq1} and \ref{lem:eq2}.

\section{Fredholm Alternative}~\label{sec:fredh}

Here we prove the Fredholm Alternative for the problem (\ref{eq:tel})--(\ref{eq:Neu}). From Section~\ref{sec:equiv}
we know how  the Fredholm and the index properties of the operators of the problems  (\ref{eq:tel})--(\ref{eq:Neu}) and (\ref{eq:walk1})--(\ref{eq:Neu1}) are
related to each other. More precisely, the following lemma is true.

 \begin{lem}\label{lem:dimker}
Suppose $\ga \ge 1$ and $\mu\ne 0$.
Then 

(i) $\dim\ker(L_{WS})=\dim\ker(L_{TE})$.

(ii) $\dim\ker(L_{WS}^*)=\dim\ker(L_{TE}^*)$.
\end{lem}

The Fredholm solvability for the periodic-Neumann problem for the telegraph equation is now a straightforward consequence of our Fredholm result
for the corresponding  correlated random walk problem.
 \begin{thm}\cite[Theorem~1]{KmRe1}\label{thm:fredh_walk}
Suppose $\ga \ge 1$ and $\mu\ne 0$.
Then we have:

(i) $L_{WS}$ is a Fredholm operator of index zero
from $V_b^{\gamma}$ into $W_d^{\gamma}$.

(ii) The image of $L_{WS}$ is the set of all $(g,g) \in W^\ga$ such that
\begin{equation}\label{eq:ortog1}
\int\limits_0^{\frac{2\pi}{\om}}
\int\limits_0^1g(x,t)(\tilde v(x,t)+\tilde w(x,t))\,dx\,dt=0
\quad\mbox{for \,all}\quad (\tilde v,\tilde w)\in\ker(\tilde L_{WS}).
\end{equation}
\end{thm}

\begin{rem}\label{rem:ortog}{\rm
As it follows from the proof of \cite[Theorem~1]{KmRe1}, 
the kernel of the problem  (\ref{eq:walk1})--(\ref{eq:Neu1}) does not depend on $\ga\ge 1$ and,
given $\ga>1$,
 all $V_b^1$-solutions to the problem (\ref{eq:walk1})--(\ref{eq:Neu1}) with $f\in W_d^{\gamma}$
necessarily belong to $V_b^\ga$
(a smoothing effect). 
This implies, in particular, that, if we have (\ref{eq:ortog1}) for $f\in W_d^{\gamma}$ with $\ga\ge 2$, 
then we automatically have 
\begin{equation}
\int\limits_0^{\frac{2\pi}{\om}}
\int\limits_0^1\d_t^sg(x,t)(\tilde v(x,t)+\tilde w(x,t))\,dx\,dt=0
\,\,\mbox{for all}\,\, (\tilde v,\tilde w)\in\ker(\tilde L_{WS}) \mbox{ and } 0\le s\le\ga.\nonumber
\end{equation}

Note that the smoothing effect does not occur for
the corresponding initial-boundary problem with Neumann boundary conditions (see~\cite{Hillen,Km,Lyulko}). 
}
\end{rem}

We are prepared to formulate our main result.

 \begin{thm}\label{thm:fredh_tel}
Suppose $\ga \ge 2$ and $\mu\ne 0$.
Then we have:

(i) 
$L_{TE}$ is a Fredholm operator of index zero
from $U_b^{\gamma}$ into $H^{0,\gamma-1}$.

(ii) The image of $L_{TE}$ is the set of all $f \in H^{0,\ga-1}$ such that
\begin{equation}\label{eq:ortog2}
\int\limits_0^{\frac{2\pi}{\om}}
\int\limits_0^1f(x,t)u(x,t)\,dx\,dt=0
\,\,\mbox{for all}\,\, u\in\ker(\tilde L_{TE}).
\end{equation}
\end{thm}

Theorem~\ref{thm:fredh_tel} follows directly from Theorems~\ref{thm:eq} and~\ref{thm:fredh_walk} and Lemma~\ref{lem:dimker}
(see also Remark~\ref{rem:ortog}).

\begin{rem}{\rm
Like the problem (\ref{eq:walk1})--(\ref{eq:Neu1}), one can observe a similar smoothing effect  for the problem
(\ref{eq:tel})--(\ref{eq:Neu}): The kernel of the problem does not depend $\ga$ and, given $\ga>2$,   all $U_b^2$-solutions 
to problem (\ref{eq:tel})--(\ref{eq:Neu}) with $f\in H^{0,\ga-1}$, necessarily belong to $U_b^\ga$.
}
\end{rem}

\begin{rem}\label{rem:perturb}{\rm
Since the set of Fredholm operators is open,  the  conclusion of Theorem \ref{thm:fredh_tel}  survives under 
sufficiently small  perturbations of the operator $L_{TE}\in\LL(U^\ga;H^{0,\ga-1})$.
The theorem remains true, if, instead of $L_{TE}$, we consider the operator of the problem (\ref{eq:tel_perturb}), (\ref{eq:per}), (\ref{eq:Neu}) with
$\nu$ and $\al$  such that 
$\nu(x,t)u_t\in H^{0,\ga}$, $\nu(x,t)$ is sufficiently small perturbation of a nonzero constant,
$\al(x,t)u\in H^{0,\ga}$, and $\al(x,t)$ is sufficiently small.
 The Fredholm property of the 
slightly perturbed problem is fulfilled independently of whether or not the perturbed problems 
(\ref{eq:tel})--(\ref{eq:Neu}) and (\ref{eq:walk1})--(\ref{eq:Neu1}) remain equivalent.
}
\end{rem}

\section*{Acknowledgments}
I thank Lutz Recke for bringing the paper~\cite{Hillen} to my attention.

\end{document}